\documentclass[10pt]{amsart}
\usepackage[mathscr]{eucal}
\usepackage{times}
\usepackage{amsmath,amssymb,latexsym,amscd}
\usepackage{mathrsfs}   
\usepackage{hyperref}
\usepackage{graphicx}
\usepackage[all,cmtip]{xy}
\usepackage{fancyhdr}

\DeclareMathOperator{\Mod}{Mod}

\DeclareMathOperator{\Gdim}{Gdim}

\newtheorem{thm}{Theorem}[section]

\newtheorem{prop}[thm]{Proposition}
\newtheorem{lem}[thm]{Lemma}
\theoremstyle{definition}
\newtheorem{dfn}[thm]{Definition}

\begin{document}


\title{A Note on Gabriel Dimension for Idioms }
\author{Angel Zald\'\i var Corichi}
\email{zaldivar@matem.unam.mx}
\address{Instituto de Matem\'aticas, Ciudad Universitaria, UNAM, M\'exico, D. F., 04510, M\'exico.}

\begin{abstract}
The aim of this note is to illustrate that the definition and construction of the Gabriel dimension for modular lattices in the sense of \cite{1} is the same as the module case in the document \cite{2}. 
\end{abstract}
 \maketitle

\section{Introduction}\label{sec:sec1}

The following definition is taken, literally, from \cite{1}, top of the page 135:
\begin{quote}
\lq\lq Let be $A$ a modular upper-continuous lattice with $0$ and $1$.
We define the \emph{Gabriel dimension} of $A$, denoted by $\Gdim(A)$, using transfinite recursion.
We put $\Gdim(A)=0$ if and only if $A=\left\{0\right\}$. Let $\alpha$ be a nonlimit ordinal and assume that the Gabriel dimension $\Gdim(A')=\beta$ has already been defined for lattices with $\beta<\alpha$. We say that $A$ is it $\alpha-simple$ is for each $a\neq 0$ in $A$ we have $\Gdim [0,a]\nless\alpha$ and $\Gdim[a,1]<\alpha$. We then say that $\Gdim(A)=\alpha$ if $\Gdim (A)\nless\alpha$ but for every $a\neq 1$ in $A$ there exist a $b>0$ such that $[a,b]$ is $\beta$-simple for some $\beta\leq\alpha$.\rq\rq
\end{quote}

In the second paragraph of the same page, it is stated the following:
\begin{quote} \lq\lq Consider $a\in A$. If $\Gdim(0,a)=\alpha$ then we say that $\alpha$ is the \emph{Gabriel dimension} of $a$ and we write $\Gdim (a)=\alpha$. If $[0,a]$ is $\alpha$-simple then $a$ is said to be an $\alpha$-simple element of $A$.\rq\rq
\end{quote}

We will rewrite this definition in the idiom context, mimicking the construction of Gabriel dimension, in the module category, given in \cite{2}. Basically, the proofs are the same as in \cite{2}.
In fact, the two constructions are related via the slicing technique, for more details about the slicing and relation with dimension in module categories and lattices we refer to \cite{2} and \cite{6}.

\section{Gabriel dimension for idioms}

To begin with, fix an idiom $A$ (that is a complete, modular, upper-continuous lattice ), let $[a,b]=\{x\in A\mid a\leq x\leq b\}$ the \emph{interval} of $a\leq b$. Denote by $\EuScript{I}(A)$ the set of all intervals of $A$ and by $\EuScript{O}(A)=\EuScript{O}$ the set of all \emph{trivial} intervals, that is, for an element $a\in A$ the trivial interval of it is $[a,a]=\{a\}$. Next we recall the definition of the Gabriel dimension for an idiom.

An interval $[a,b]$ is \emph{simple} if $[a,b]=\{a,b\}$ observe now that this is equivalent to say: \[\text{ An interval } [a,b] \text{ is } \emph{simple} \text{ if for every } a\leq x\leq b \text{ one has } a=x \text{ or } b=x \] and immediately this is also equivalent to :
\[\text{ An interval } [a,b] \text{ is } \emph{simple} \text{ if for every } a\leq x\leq b \text{ one has } [a,x]\in\EuScript{O}(A) \text{ or } [x,b]\in\EuScript{O}(A) \]
with this in mind the relative version of the $\EuScript{O}(A)-$\emph{simple} is direct, that is, given a set of intervals $\mathcal{B}\subseteq\EuScript{I}(A)$ an interval $[a,b]$ is $\mathcal{B}$-\emph{simple} if for every $a\leq x\leq b$ one has $[a,x]\in\mathcal{B}$ or $[x,b]\in\mathcal{B}$. Observe now that this produce an operation in the set all sets of intervals on $A$ more over this operation is defined in a particular kind of sets of intervals. As in the case with module classes closed under certain kind of operations one introduce the following, mimicking the module idea:

Given two intervals $I=[a,b]$ and $J=[a',b']$, we say that $I$ is a \emph{subinterval} of $J$, denoted by $I\rightarrow J$, if $I=[a,b]$ and $J=[a',b']$ with $a'\leq a\leq b\leq b'$ in $A$. We say that $J$ and $I$ are \emph{similar}, denoted by  $J\sim I$, if there are $l,r\in A$ with associated intervals \[L=[l,l\vee r]\;\;\;\; [l\wedge r,r]=R\] where $J=L$ and $I=R$ or $J=R$ and $I=L$. Clearly, this a reflexive and symmetric relation. Moreover, if $A$ is modular, this relation is just the canonical lattice isomorphism between $L$ and $R$. 

A set of intervals $\mathcal{A}\subseteq {\EuScript I}(A)$ is \emph{abstract} if is not empty and it is closed under $\sim$, that is, 
\[J\sim I\in\mathcal{A}\Rightarrow J\in\mathcal{A}.\] 
An abstract set $\mathcal{B}$ is a \emph{basic} set of intervals if it is closed by subintervals, that is, 
\[J\rightarrow I\in\mathcal{B}\Rightarrow J\in\mathcal{B}\] 
for all intervals $I,J$. A set of intervals $\mathcal{C}$ is a \emph{congruence} set if it is basic and closed under abutting intervals, that is, 
\[[a,b], [b,c]\in \mathcal{C}\Rightarrow [a,c]\in\mathcal{C}\]
for elements $a,b,c\in A$. A basic set of intervals $\mathcal{B}$ is a \emph{pre-division} set if $$\forall\; x\in X\left[[a,x]\in\mathcal{B}\Rightarrow [a,\bigvee X]\in\mathcal{B}\right]$$ for each $a\in A$ and $X\subseteq [a,\bar{1}]$. A set of intervals $\mathcal{D}$ is a \emph{division} set if it is a congruence set and a pre-division set.
Put $\mathcal{D}(A)\subseteq\mathcal{C}(A)\subseteq\EuScript{B}(A)\subseteq\EuScript{A}(A)$ the set of all division, congruence, basic and abstract set of intervals in $A$. This gadgets can be understood like certain classes of modules in a module category $R$-$\Mod$, that is, classes closed under isomorphism, subobjects, extensions and coproducts. From this point of view $\EuScript{C}(A)$ and $\EuScript{D}(A)$ are the idioms analogues of the Serre classes and the torsion (localizations) classes in module categories.

Is straightforward to see that $\EuScript{B}(A)$ and $\EuScript{A}(A)$ are frames also $\mathcal{D}(A)$ and $\mathcal{C}(A)$ are frames too this is not directly, the details of these are in \cite{5}. Let be $\EuScript{S}mp(\mathcal{B})$ the set of all $\mathcal{B}$-simples intervals, this set is basic provided $\mathcal{B}$ is a basic set. To define the gabriel dimension of an idiom, specifically the Gabriel dimension of an interval we need to produce a \emph{filtration}.This filtration is related with the simples and with \emph{critical} intervals that is, let be $\mathcal{B}\in\EuScript{B}(A)$ and denote by $\EuScript{C}rt(\mathcal{B})$ the set of intervals such that for all $a\leq x\leq b$ we have $a=x$ or $[x,b]\in\mathcal{B}$; this is the set of all $\mathcal{B}$-\emph{critical} intervals. Note that $\EuScript{S}mp(\EuScript{O})=\EuScript{C}rt(\EuScript{O})$ and $\EuScript{C}rt(\mathcal{B})\leq \EuScript{S}mp(\mathcal{B})$.

As we mention before the set $\mathcal{D}(A)$ is a frame in particular is a complete lattice therefore for any basic set $\mathcal{B}$ there exists the least division set that contains it $\mathcal{D}vs(\mathcal{B})$, this description set up an operation in the frame of basic sets of intervals, that is, a function $\EuScript{K}pr\colon\EuScript{B}(A)\rightarrow\EuScript{B}(A)$ such that $\mathcal{B}\leq\EuScript{K}pr(\mathcal{B})$, $\EuScript{K}pr(\mathcal{B})\leq \EuScript{K}pr(\mathcal{A})$ whenever $\mathcal{B}\leq\mathcal{A}$ and $\EuScript{K}pr(\mathcal{B}\cap\mathcal{A})=\EuScript{K}pr(\mathcal{B})\cap\EuScript{K}pr(\mathcal{A})$, this kind of functions are called \emph{pre-nucleus} a \emph{nucleus} is an idempotent pre-nucleus. The $\mathcal{D}vs$ construction is a nucleus on $\EuScript{B}(A)$ with this we can set up $\EuScript{G}ab:=\mathcal{D}vs\circ\EuScript{C}rt$ this is the Gabriel pre-nucleus of $A$ and one can prove that $\mathcal{D}vs\circ\EuScript{C}rt=\EuScript{G}ab=\mathcal{D}vs\circ\EuScript{S}mp$. 
We can iterate $\EuScript{G}ab$ over all ordinals to obtain a chain of division sets $\EuScript{O}(A)\leq\EuScript{G}ab(\EuScript{O})\leq\ldots\leq\EuScript{G}ab^{\alpha}_(\EuScript{O})\leq\ldots$ where $\EuScript{G}ab^{\alpha}(\EuScript{O})$ is defined by $\EuScript{G}ab^{\alpha+1}(\EuScript{O}):=\EuScript{G}ab(\EuScript{G}ab^{\alpha}(\EuScript{O}))$ and $\EuScript{G}ab^{\alpha}(\EuScript{O})=\mathcal{D}vs(\bigcup\{\EuScript{G}ab^{\beta}(\EuScript{O})\mid\beta<\alpha\})$ for non-limit and limit ordinals.
Now with this filtration we can define the \emph{Gabriel dimension} of an interval $[a,b]$ to be the extended ordinal $G(a,b)\leq\alpha$ if and only if $[a,b]\in\EuScript{G}ab^{\alpha}(\EuScript{O})$. The central objective of this short note is to illustrate that the construction of \cite{1} produce this filtration but with another point of view. For a more detail treatment of this construction and related topics with dimension and inflator theory the reader is refer to \cite{6}, \cite{7} and \cite{14}.

\begin{dfn}\label{l1}
We define the \emph{Gabriel dimension}, $\Gdim$ of $[a,b]$ as follows:
\begin{enumerate}
\item $\Gdim(a,b)=0\Leftrightarrow a=b$.
\item $\Gdim(a,b)=\alpha'\Leftrightarrow \Gdim(a,b)\nleq \alpha$ and $$\left(\forall a\leq x< b\right)\left[\exists x<y\leq b \right]\left[\exists \beta\leq\alpha'\right]\left[[x,y]\text{ is }\beta\text{-simple}\right],$$
for ordinals $\alpha$ and $\alpha'$ its successor.
\item $\Gdim(a,b)=\lambda\Leftrightarrow \left(\forall a\leq x< b\right)\left(\exists x<y\leq b\right)\left[\exists \beta<\lambda\right]\left[[x,y]\text{ is }\beta\text{-simple}\right],$
for limit ordinals $\lambda$.
\end{enumerate}

Here, $\beta\text{-simple}$ means that for the successor ordinal $\beta$, the interval $[a,b]$ is $\beta$\emph{-simple} if: $$\left(\forall a<x\leq b\right)\left[\Gdim(a,x)\nless\beta \text{ and } \Gdim(x,b)<\beta\right]$$
\end{dfn}

Following \cite{2}, we say that the only $0$-simple and  $\lambda$-simple intervals, for all  limit ordinals $\lambda$,  are the trivial ones, that is, $\EuScript{O}(A)$. 
Then, condition (3) of Definition \ref{l1} is reinterpreted as: 
$$ \Gdim(a,b)=\lambda\Leftrightarrow \left(\forall a\leq x< b\right)\left(\exists x<y\leq b\right)\left[\exists \beta\leq\lambda\right]\left[[x,y]\text{ is }\beta\text{-simple}\right].$$

Next we make these definitions accumulative. Following \cite{2}, define the set $\EuScript{S}[\alpha]$ of $\alpha$-simple intervals, with $\alpha$ an ordinal, as
 $$[a,b]\in\EuScript{S}[\alpha]\Leftrightarrow\left(\forall a<x\leq b\right)\left(\Gdim(a,x)\nleq\alpha\text{ and } \Gdim(x,b)\leq\alpha\right),$$ 
and then proceed step by step as follows:
\begin{enumerate} 
\item $\EuScript{D}(0)=\EuScript{O}(A)$.
\item $\EuScript{D}(\alpha')=\EuScript{D}(\alpha)\cup \EuScript{S}[\alpha]$
\item $\EuScript{D}(\lambda)=\bigcup\left\{\EuScript{D}(\alpha)\;|\;\alpha<\lambda\right\}$,
\end{enumerate}
for each ordinal $\alpha$ and limit ordinal $\lambda$.

In Definition \ref{l1} there is a (strange) quantification $(\exists\beta)$ in items (2) and (3).  To deal with this quantification and make everything more clear, we introduce the following definitions:

\begin{dfn}
\label{l2}
For each $\mathcal{C}\subseteq\EuScript{I}(A)$, set: $$[a,b]\in(\forall\exists)(\mathcal{C})\Leftrightarrow\left(\forall a\leq x<b\right)\left(\exists x<y\leq b\right)\left[[x,y]\in\mathcal{C}\right].$$
\end{dfn}

Immediately one observes that, if $\mathcal{C}$ is basic then $(\forall\exists)(\mathcal{C})=\EuScript{D}vs(\mathcal{C})$.  Note also that the operator $(\forall\exists)(\_)$ is monotone. (For the details about the $\EuScript{D}vs$-construction see \cite{5}-Theorem 5.6)
With this we redefine:

\begin{dfn}[$\EuScript{L}\text{-construction}$]
\label{l3}
For each interval $[a,b]$ and for each ordinal $\alpha$ and limit ordinal $\lambda$, we set:
\begin{enumerate}
\item $[a,b]\in\EuScript{L}[0]\Leftrightarrow a=b$,
\item $[a,b]\in\EuScript{L}[\alpha']\Leftrightarrow [a,b]\in\left(\forall\exists\right)(\EuScript{D}(\alpha'))\text{ and } [a,b]\notin\EuScript{L}(\alpha)$,
\item $[a,b]\in\EuScript{L}[\lambda]\Leftrightarrow [a,b]\in(\forall\exists)(\EuScript{D}(\lambda))$,
\end{enumerate}
where: 
\begin{enumerate}
\item $\EuScript{L}(0)=\EuScript{O}(A)$
\item $\EuScript{L}(\alpha')=\EuScript{L}(\alpha)\cup\EuScript{L}[\alpha']$
\item $\EuScript{L}(\lambda)=\bigcup\left\{\EuScript{L}(\alpha)\;|\;\alpha<\lambda\right\}\cup\EuScript{L}[\lambda]$,
\end{enumerate}
and 
\begin{enumerate}
\item $\EuScript{D}(0)=\EuScript{O}(A)$
\item $\EuScript{D}(\alpha')=\EuScript{D}(\alpha)\cup\EuScript{S}[\alpha]$
\item $\EuScript{D}(\lambda)=\bigcup\left\{\EuScript{D}(\alpha)\;|\;\alpha<\lambda\right\}$,
\end{enumerate}
where again $$ [a,b]\in\EuScript{S}[\alpha]\Leftrightarrow\left(\forall a<x\leq b\right)\left[[a,x]\notin\EuScript{L}(\alpha)\text{ and }[x,b]\in\EuScript{L}(\alpha)\right],$$
is the accumulative version of $\alpha$-simplicity. Here $\EuScript{L}[\alpha]$ is the set of all intervals with $\Gdim(a,b)=\alpha$ and $\EuScript{L}(\alpha)=\bigcup\left\{\EuScript{L}[\beta]\;|\;\beta\leq\alpha\right\}$ the set of intervals with Gabriel dimension $\Gdim(a,b)\leq\alpha$.

\end{dfn}

\begin{lem}
\label{l4}
For each ordinal $\alpha$ we have $$\EuScript{L}(\alpha')=\EuScript{L}(\alpha)\cup(\forall\exists)(\EuScript{D}(\alpha')).$$
\end{lem}

\begin{proof}
For each interval $[a,b]$ we have : 
\begin{align*}
[a,b]\in\EuScript{L}(\alpha')& \Leftrightarrow [a,b]\in\EuScript{L}(\alpha)\text{ or }\EuScript{L}[\alpha']\\
& \Leftrightarrow [a,b]\in\EuScript{L}(\alpha)\text{ or }\left([a,b]\in(\forall\exists)(\EuScript{D}(\alpha'))\text{ and } [a,b]\notin\EuScript{L}(\alpha) \right)\\
&\Leftrightarrow [a,b]\in\EuScript{L}(\alpha)\text{ or }[a,b]\in(\forall\exists)(\EuScript{D}(\alpha')).
\end{align*}
\end{proof}

\begin{dfn}[Accumulative $\EuScript{L}\text{-construction}$]\label{l5}
For each ordinal $\alpha$ and limit ordinal $\lambda$, introduce:
$$
\begin{array}{rclcrcl}
\EuScript{L}(0) & =&  \EuScript{O}(A) &  & \EuScript{D}(0)  & = & \EuScript{O}(A)\\
\EuScript{L}(\alpha')  & = & \EuScript{L}(\alpha)\cup(\forall\exists)(\EuScript{D}(\alpha'))&  & \EuScript{D}(\alpha')& = & \EuScript{D}(\alpha)\cup\EuScript{S}[\alpha]\\
\EuScript{L}(\lambda)   & = &  \bigcup\left\{\EuScript{L}(\alpha)\;|\alpha<\lambda\;\right\}\cup(\forall\exists)(\EuScript{D}(\lambda)) & &    \EuScript{D}(\lambda)&=&\bigcup\left\{\EuScript{D}(\alpha)\;|\;\alpha<\lambda\right\}\\
\end{array}
$$
Where, again, in the step: 
$$ [a,b]\in\EuScript{S}[\alpha]\Leftrightarrow\left(\forall a<x\leq b\right)\left[[a,x]\notin\EuScript{L}(\alpha)\text{ and }[x,b]\in\EuScript{L}(\alpha)\right]$$ for each interval.
\end{dfn}

As Simmons says, this is getting easier to read, and the construction gives two ascending chains of sets of intervals $$\EuScript{L}(0)\subseteq\cdots\subseteq\EuScript{L}(\alpha)\subseteq\cdots\qquad \qquad\text{and}\qquad \qquad\EuScript{D}(0)\subseteq\cdots\subseteq\EuScript{D}(\alpha)\subseteq\cdots,$$ 
and the aim of this note is to show that $\EuScript{L}(-)$ produces the Gabriel filtration in $A$ for $\EuScript{O}(A)$, that is, 
$$\EuScript{L}(\alpha)=\EuScript{G}ab^{\alpha}(\EuScript{O}(A)).$$ 
Then, we must first show:

\begin{thm}\label{l6}
For each ordinal $\alpha$, the collection $\EuScript{L}(\alpha)$ is a division set in $A$.
\end{thm}

\begin{proof}
Clearly the set $\EuScript{L}(\alpha)$ is an abstract set. 
Now, for the proofs of the basic congruences and $\bigvee$-closed properties, we invoke  Proposition 3.4.1, Corollary 3.4.2, and 3.4.3 of \cite{1}.
\end{proof}

\begin{dfn}\label{l7}
For each ordinal $\alpha$ let be $$\EuScript{C}(\alpha)=\EuScript{L}(\alpha)\cup\EuScript{S}[\alpha]$$
where $\EuScript{S}[\alpha]$ is the set of all $\alpha'$-simple intervals. 
\end{dfn}

\begin{lem}
\label{l8}
For each ordinal $\alpha$,  $$\EuScript{C}(\alpha)=\EuScript{C}rt(\EuScript{L}(\alpha)).$$
\end{lem}

\begin{proof}
We must show: $$[a,b]\in\EuScript{C}(\alpha)\Leftrightarrow\forall\; a\leq x\leq b\;\; a=x\text{ or } [x,b]\in\EuScript{L}(\alpha)$$.

Assuming $[a,b]\in\EuScript{C}(\alpha)$, then by definition \ref{l7}, if $[a,b]\in\EuScript{L}(\alpha)$, then the conclusion is clear from the fact that $\EuScript{L}(\alpha)$ is in particular basic.
If $[a,b]\in\EuScript{S}[\alpha]$, consider $a\leq x\leq b$; thus by definition of this set we have that $[x,b]\in\EuScript{L}(\alpha)$.

Reciprocally, if $[a,b]\in\EuScript{S}[\alpha]$ there is nothing to prove. Thus, suppose $[a,b]\notin\EuScript{S}[\alpha]$, then there is a $a<x\leq b$ such that $[a,x]\in\EuScript{L}(\alpha)$  or $[x,b]\notin\EuScript{L}(\alpha)$.  But the condition says that $[x,b]\in\EuScript{L}(\alpha)$ and $\EuScript{L}(\alpha)$ is a congruence set, thus $[a,b]\in\EuScript{L}(\alpha)$. 
\end{proof}

\begin{prop}\label{l9}
We have: $$\EuScript{D}(\alpha)\subseteq\EuScript{L}(\alpha)$$ for each ordinal $\alpha$.

\end{prop}

\begin{proof}
By induction, the case $\alpha=0$ being obvious because, $\EuScript{D}(0)=\EuScript{O}(A)=\EuScript{L}(0)$ by definition of these sets. 
For the step $\alpha\mapsto\alpha'$, suppose that $[a,b]\in\EuScript{D}(\alpha')$. The definition of this set gives two possibilities: First, if $[a,b]\in\EuScript{D}(\alpha)$ then from the induction hypothesis $[a,b]\in\EuScript{L}(\alpha)\EuScript{L}(\alpha')$.
Now, if $[a,b]\notin\EuScript{L}(\alpha)$ then $[a,b]\in\EuScript{S}[\alpha]$ and in this case we will show that $[a,b]\in(\forall\exists)(\EuScript{D}(\alpha'))$ and using $(\forall\exists)(\EuScript{D}(\alpha'))\subseteq\EuScript{L}(\alpha')$, we will be done.  To prove our claim, consider $a\leq x<b$. We will produce a $x<y\leq b$ with $[x,y]\in\EuScript{D}(\alpha')$ and show that $y=b$ is the required element.
If $a=x$, there is nothing to prove. If $a\neq x$ then $[a,b]\in\EuScript{S}[\alpha]$ gives $[a,x]\notin\EuScript{L}(\alpha)$ and $[x,b]\in\EuScript{L}(\alpha)$. If  $[a,x]\notin\EuScript{L}(\alpha)$, the induction hypothesis gives $[x,b]\in\EuScript{L}(\alpha)\subseteq\EuScript{D}(\alpha)\subseteq\EuScript{D}(\alpha')$, and we are done.

Now, for the limit case $\lambda$ we have $$\EuScript{D}(\lambda)=\bigcup\left\{\EuScript{D}(\alpha)\;|\;\alpha<\lambda\right\}\subseteq\bigcup\left\{\EuScript{L}(\lambda)\;|\;\alpha<\lambda\right\}\subseteq\EuScript{L}(\lambda),$$  where the inclusion $\bigcup\left\{\EuScript{D}(\alpha)\;|\;\alpha<\lambda\right\}\subseteq\bigcup\left\{\EuScript{L}(\lambda)\;|\;\alpha<\lambda\right\}$ is by the induction hypothesis. 
\end{proof}

From Proposition \ref{l9}, Lemma  \ref{l8} and Definition  \ref{l7}, it follows that $$\EuScript{L}(\alpha)\cup\EuScript{S}[\alpha]=\EuScript{C}(\alpha)=\EuScript{C}rt(\EuScript{L}(\alpha)).$$ 
From the fact that $\EuScript{C}(\alpha)$ is basic upon applying $\EuScript{G}ab$ we have $\EuScript{G}(\EuScript{L}(\alpha))=\EuScript{D}vs(\EuScript{C}(\alpha))=(\forall\exists)(\EuScript{C}(\alpha))$ since the two operators $\EuScript{D}vs$ and $(\forall\exists)$ agree on basic sets.
All this is  summarized in the following
\begin{thm}
\label{l10}
With the above notation we have $$\EuScript{G}ab(\EuScript{L}(\alpha))=\EuScript{L}(\alpha')$$ for each ordinal $\alpha$.

\end{thm}

\begin{proof}
From Proposition \ref{l9} and the definition of $\EuScript{D}(\alpha')$ we have $\EuScript{C}(\alpha)=\EuScript{L}(\alpha)\cup\EuScript{S}[\alpha]\subseteq\EuScript{L}(\alpha)\cup\EuScript{D}(\alpha')\subseteq\EuScript{L}(\alpha')$.  It follows that $\EuScript{G}ab(\EuScript{L}(\alpha))=\EuScript{D}vs(\EuScript{C}(\alpha))\subseteq\EuScript{L}(\alpha')$ by the remark before this theorem and the fact that $\EuScript{L}(\alpha')$ is a division set.
For other inclusion we have $\EuScript{D}(\alpha')=\EuScript{D}(\alpha)\cup\EuScript{S}[\alpha]\subseteq\EuScript{L}(\alpha)\cup\EuScript{S}[\alpha]=\EuScript{C}(\alpha)$ again by Proposition \ref{l9}.  From the monotonicity of $(\forall\exists)(\_)$ it follows that $(\forall\exists)(\EuScript{D}(\alpha'))\subseteq(\forall\exists)(\EuScript{C}(\alpha))=\EuScript{G}ab(\EuScript{L}(\alpha))$, and then $\EuScript{L}(\alpha')=\EuScript{L}(\alpha)\cup(\forall\exists)(\EuScript{D}(\alpha'))\subseteq\EuScript{G}ab(\EuScript{L}(\alpha))$ since $\EuScript{G}ab$ is an inflator.
\end{proof}

We can now prove the main result of this note:

\begin{thm}
\label{l11}
With the same notation we have $$\EuScript{L}(\alpha)=\EuScript{G}ab^{\alpha}(\EuScript{O})$$ for each ordinal $\alpha$. Here $\EuScript{O}=\EuScript{O}(A)$.
\end{thm}

\begin{proof}
By induction on $\alpha$ , the base case $\alpha=0$, being clear. The induction step is just Theorem \ref{l10}. For the limit case $\lambda$ let $\mathfrak{L}=\bigcup\left\{\EuScript{L}(\alpha)\;|\; \alpha<\lambda\right\}$. Since $\EuScript{L}(\alpha)$ is basic for each ordinal, then $\mathfrak{L}$ is also basic. Thus, the induction hypothesis gives $$\EuScript{G}ab^{\lambda}(\EuScript{O})=\EuScript{D}vs(\mathfrak{L}),$$ 
and by the accumulative $\EuScript{L}$-construction $$\EuScript{L}(\lambda)=\mathfrak{L}\cup(\forall\exists)(\EuScript{D}(\lambda))$$ 
and
\begin{align*}
(\forall\exists(\EuScript{D}(\lambda)))&=(\forall\exists)(\bigcup\left\{\EuScript{D}(\alpha)\;|\;\alpha<\lambda\right\})\\
&=(\forall\exists)(\bigcup\left\{\EuScript{D}(\alpha')\;|\;\alpha<\lambda\right\})\subseteq (\forall\exists)(\bigcup\left\{\EuScript{L}(\alpha')\;|\;\alpha<\lambda\right\})=(\forall\exists)(\mathfrak{L})\\
&=\EuScript{D}vs(\mathcal{L})
\end{align*}
where the first equality is the definition of $\EuScript{D}(\lambda)$ in the limit case, the second equality is because the construction $\EuScript{D}(-)$ is an ascending chain. The inclusion in the second row is from Theorem \ref{l10} and the monotonicity of $(\forall\exists)(\_)$.  The last equality is because the operators $\EuScript{D}vs$ and $(\forall\exists)$ agree on basic sets. Finally, with this and the description of $\EuScript{L}(-)$ in the limit case we conclude that $$\EuScript{L}(\lambda)=\mathfrak{L}\cup\EuScript{D}vs(\mathfrak{L})=\EuScript{D}vs(\mathfrak{L})=\EuScript{G}ab^{\lambda}(\EuScript{O})$$

\end{proof}

\bibliographystyle{model1-num-names}

\end{document}